\documentclass[11pt,a4paper,dvipsnames,eqrno]{article}
\usepackage[utf8]{inputenc}
\usepackage{a4wide}

\usepackage{graphicx}
\usepackage{amssymb}
\usepackage{amsmath}
\usepackage{amsthm}
\usepackage{thmtools}
\usepackage{mathtools}
\usepackage{multirow}
\usepackage{blkarray}
\usepackage[dvipsnames]{xcolor}
\usepackage{etoolbox}
\usepackage{enumerate}
\usepackage{enumitem}
\usepackage{verbatim}
\usepackage{comment}
\usepackage{array}
\usepackage{dsfont}
\usepackage{soul}

\usepackage{relsize}
\usepackage[pdftex, colorlinks, linkcolor=blue, citecolor=blue, urlcolor=blue]{hyperref}
\usepackage[nameinlink]{cleveref}
\hypersetup{linkcolor=MidnightBlue, citecolor=PineGreen, urlcolor=Emerald}
\usepackage{pgfplots}
\usepackage[margin=1cm]{caption}
\usepackage{subcaption}
\pgfplotsset{width=7cm, compat=1.10}
\newcommand{\R}{\mathbb{R}}

\newcommand{\C}{\mathbb{C}}
\newcommand{\CP}{\mathbb{CP}}
\newcommand{\N}{\mathbb{N}}
\newcommand{\Q}{\mathbb{Q}}
\newcommand{\1}{\mathds{1}}
\newcommand{\pr}{p_{\mathbb{R}}}
\newcommand{\pc}{p_{\mathbb{C}}}
\newcommand{\qr}{q_{\mathbb{R}}}
\newcommand{\qc}{q_{\mathbb{C}}}

\newcommand{\invp}{B^\iota_d}
\newcommand{\invptwo}{B^\iota_2}
\newcommand{\invpthree}{B^\iota_3}

\DeclareMathOperator{\Wd}{W_2^2}
\DeclareMathOperator{\cost}{cost}
\DeclareMathOperator{\supp}{supp}
\DeclareMathOperator{\Sym}{Sym}
\DeclareMathOperator{\Aut}{Aut}
\DeclareMathOperator{\Wdeg}{WDdegree}

\theoremstyle{definition}
\newtheorem{theorem}{Theorem}[section]
\newtheorem{corollary}[theorem]{Corollary}
\newtheorem{proposition}[theorem]{Proposition}
\newtheorem{definition}[theorem]{Definition}

\newtheorem{example}[theorem]{Example}
\newtheorem{lemma}[theorem]{Lemma}
\newtheorem{remark}[theorem]{Remark}

\title{\textbf{The algebraic degree of the Wasserstein distance}}
\author{Chiara Meroni, Bernhard Reinke, and Kexin Wang}
\date{}

\begin{document}

\maketitle

\begin{abstract}
    Given two rational univariate polynomials, the Wasserstein distance of their associated measures is an algebraic number. We determine the algebraic degree of the squared Wasserstein distance, serving as a measure of algebraic complexity of the corresponding optimization problem. The computation relies on the structure of a subpolytope of the Birkhoff polytope, invariant under a transformation induced by complex conjugation.
\end{abstract}

\section{Introduction}

Given two univariate polynomials $p,q \in \Q[z]$ of degree $d$ with simple roots, consider the optimization problem
\begin{equation}\label{eq:opt_problem}
    \hbox{minimize } \frac{1}{d} \sum_{i=1}^d \| \alpha_i - \beta_{\sigma(i)} \|^2 \quad \hbox{ over } \sigma \in \Sym(d),
\end{equation}
where $\alpha_1, \ldots, \alpha_d$ and $\beta_1,\ldots,\beta_d$ are the roots of $p$ and $q$ respectively. Here $\|\cdot\|$ denotes the Euclidean norm and $\Sym(d)$ is the set of permutations of $d$ elements. The optimal value of \eqref{eq:opt_problem} is the square of the Wasserstein distance of the measures associated to $p$ and $q$. It follows from the definition that this minimum is an algebraic number, namely it is a root of a univariate polynomial with rational coefficients. There exists a unique such polynomial that is monic and of minimal degree; we denote the minimal degree by $\Wdeg(p,q)$. We will refer to $\Wdeg(p,q)$ as the \textit{Wasserstein distance degree} of $p$ and $q$. The goal of this paper is to provide a formula for $\Wdeg(p,q)$. We propose this algebraic approach to an intrinsically metric problem in the spirit of \textit{metric algebraic geometry} \cite{BKS23:MetricAlgGeom}, confident that the interplay between the different points of view will bring new insights for both communities.

The Wasserstein distance is a notion of distance between probability distributions on a metric space. It can be thought of as the cost of turning one distribution into another, and it is therefore intimately connected to optimal transport problems \cite[Chapter 7]{Villani03:TopicsOptTrans}. The optimization problem \eqref{eq:opt_problem} is a special case of the assignment problem \cite[Chapter 17]{Schrijver03:CombinatorialOptimization}, a well-studied problem in combinatorial optimization.

Adopting a broader perspective, we can interpret \eqref{eq:opt_problem} as a polynomial optimization problem where we minimize a polynomial with rational coefficients in $2d$ variables under $2d$ constraints:
\begin{equation*}
    \hbox{minimize } f(x_1, \ldots,x_d, y_1,\ldots,y_d) \; \hbox{ subject to } f_k(x_1, \ldots,x_d) = 0, \, g_k(y_1,\ldots,y_d) = 0,
\end{equation*}
where $f(x_1, \ldots,x_d, y_1,\ldots,y_d) = \frac{1}{d} \sum_i \|x_i-y_i\|^2$ and $f_k, g_k$ are Vieta's formulae \cite{Vieta46:OperaMatematica} for $k=1,\ldots,d$, namely the equations of degree $k$ that relate the roots of $p,q$ to their coefficients. It is of general interest to compute the algebraic degree of the optimal solution of a polynomial optimization problem. Indeed, this degree encodes the \textit{algebraic complexity} required to write the optimum \textit{exactly} in terms of the input.  
This concept has been investigated in the literature in several contexts, e.g., semidefinite programming \cite{GvBR09:AlgDegreeSDP,NRS10:AlgDegSDP,RosStu10:DualitiesConvexAlgGeom,HGV23:CharacterizationAlgDegSDP}. Alternatively, when the objective function is the squared Euclidean distance \cite{DHOST16:EDdegreeAlgVariety,OttSod20:DistRealAlgVariety}, the algebraic degree of the minimum (ED degree) measures the complexity of optimizing the Euclidean distance over an algebraic variety, problem arising in many applications. Of similar fashion is the maximum likelihood degree, central in algebraic statistics \cite{HKS05:SolvingLikelihoodEqns,HuhStu14:LikelihoodGeom,Sul18:AlgStatistics}.

Algebraic techniques allow us to compute the algebraic degree of the optimization problem when the constraints are generic, see \cite{NieRan09:AlgDegPolyOpt}. However, in many relevant scenarios (e.g., those mentioned above), the imposed conditions are not generic enough. Then, \cite[Theorem 2.2]{NieRan09:AlgDegPolyOpt} provides an upper bound for the algebraic degree. Our case fits in this family of nongeneric problems, and we can conclude that $\Wdeg(p,q) \leq (d!)^2$ for any $p,q$ as in \eqref{eq:opt_problem}. As anticipated, this bound is not tight, and we will refine it in what follows.

We begin with the setup in \Cref{sec:setting} where we define the Wasserstein distance for measures, and then interpret it in terms of associated polynomials, along the lines of \cite{ACL23:OptTransportAlgHypers}. In \Cref{sec:Birkhoff}, we introduce a transformation of the Birkhoff polytope induced by complex conjugation and study its invariant subpolytope (\Cref{def:invBirkhoff}). 
We give a graph interpretation of its vertices and characterize them in \Cref{thm: vertices of gamma} and \Cref{prop:realization}. Exploiting the combinatorial information together with standard Galois theory, we obtain the formulae for the minimal polynomial and the algebraic degree of the squared Wasserstein distance in \Cref{thm:algebraic_degree,thm: bound sharp on dense set}. In \Cref{subsec:elimination}, we interpret the computation of the minimal polynomial via an elimination of ideal and describe the algorithm to compute it.

We conclude the introduction with an example.
\begin{example}\label{ex:intro}
    \begin{figure}[ht]
        \centering
        \includegraphics[width=0.35\textwidth]{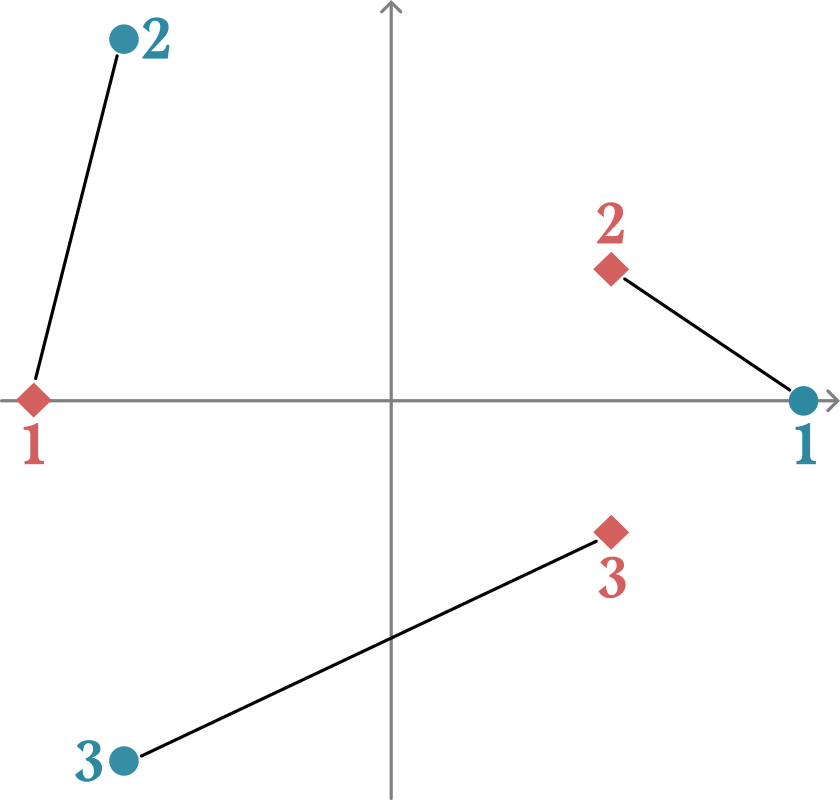}
        \caption{The roots of $p$ (blue dots) and the roots of $q$ (red diamonds), from \Cref{ex:intro}. The segments joining them represent a pairing that realizes the minimum for the Wasserstein distance of $p,q$.}
        \label{fig:ex_pq_intro}
    \end{figure}
    Consider the univariate cubic polynomials $p(z) = 2 z^3 + z^2 - 3 z -7$ and $q(z) = 3 z^3 - z^2 - 5 z + 4$. They both have one real root and a pair of complex conjugate roots. Then, $\Wd(p,q) = 2.02001392\ldots$ is the squared Wasserstein distance of $p,q$. This real number is algebraic: indeed, it is a root of the following univariate monic irreducible polynomial with rational coefficients:
    \begin{equation}\label{eq:minimal_polynomial_intro}
        \begin{gathered}
        z^{18}-\tfrac{25}{18} z^{17} -\tfrac{6611}{432} z^{16} +\tfrac{9658735}{314928} z^{15} +\tfrac{860802973}{5038848} z^{14} -\tfrac{656593649875}{816293376} z^{13} +\tfrac{129841846201105}{1586874322944} z^{12} \\
        +\tfrac{28152774676858415}{3570467226624} z^{11} -\tfrac{46036000403242535}{3570467226624} z^{10} -\tfrac{64389115881803301815}{1156831381426176} z^9 \\
        +\tfrac{987135963331795614481}{4627325525704704} z^8
        -\tfrac{188523216802141986655495}{749626735164162048} z^7 +\tfrac{570696146952180632862839}{26986562465909833728} z^6 \\
        +\tfrac{2978098003861411543837955}{40479843698864750592} z^5
        +\tfrac{674775812403532619558719717}{364318593289782755328} z^4\\
        -\tfrac{120675780215928578619735476735}{39346408075296537575424} z^3 -\tfrac{1185279955326920964318656606075}{354117672677668838178816} z^2 \\
        +\tfrac{1848261661275845551117166614375}{3187059054099019543609344} z +\tfrac{18255232646137865373249774012390625}{1835746015161035257118982144}.
    \end{gathered}
    \end{equation}
    To the polynomials $p,q$, after pairing their roots as in \Cref{fig:ex_pq_intro}, we can associate a graph $\Gamma$, namely the bottom-left graph in \Cref{fig:square_deg3}. The graph can be interpreted as a vertex of a Birkhoff-type polytope, constructed exploiting the root pattern of $p$ and $q$. By looking at the graph and using \Cref{table:WDdegrees}, we get the upper bound of $\frac{(3!)^2}{2} = 18$ for the algebraic degree of $\Wd(p,q)$, which is actually an equality by the irreducibility of \eqref{eq:minimal_polynomial_intro}.
\end{example}

\section{Finitely supported probability measures}\label{sec:setting}

The Wasserstein distance allows to put a distance between the space of probability measures on a metric space. In \cite{ACL23:OptTransportAlgHypers}, it has been used to introduce a new metric on hypersurfaces of degree $d$ in $\CP^n$.
We will recall the definition of Wasserstein distance for finitely supported measures and relate the Wasserstein distance of univariant polynomials with an optimization problem on the Birkhoff polytope.
In this section, we consider univariate polynomials with coefficients in $\mathbb{C}$.
We associate probability measures to finitely supported probability measures induced by polynomials in the following way.

\begin{definition}
    Let $p \in \C[z]$ be a polynomial of degree $d$. Its associated measure $\mu(p)$ is
    \begin{equation*}
        \mu(p) = \frac{1}{d} \sum_{p(x) = 0} m(p,x) \delta_x,
    \end{equation*}
    where $m(p,x)$ is the multiplicity of $p$ at $x\in\C$.
\end{definition}
More concretely, if $p(z)$ factors as $\prod_{i=1}^d (z - \alpha_i)$, then $\mu(p) = \frac{1}{d} \sum_{i=1}^d \delta_{\alpha_i}$. In particular, if $p$ has no multiple roots, then $\mu(p)$ is the uniform counting measure on the roots of $p$. We will be mostly interested in the case where $p$ has no multiple roots.

\begin{definition}
    Let $\mu_1, \mu_2$ be finitely supported probability measures on $\C$. A \emph{transport plan} from $\mu_1$ to $\mu_2$ is a probability measure $\lambda$ on $\C \times \C$ such that
    $\pi_{1*} \lambda = \mu_1$ and $\pi_{2*} \lambda = \mu_2$. Here $\pi_{1*} \lambda$ is defined via \[
        \pi_{1*}(\lambda)[A] = \lambda[A \times \C],
    \]
    similar for $\pi_{2*}$.
    We write $\Pi(\mu_1,\mu_2)$ for the set of transport plans between $\mu_1$ and $\mu_2$.
    The \emph{cost} of a transport plan is given by
    \begin{equation*}
        \cost(\lambda) = \int \|x - y \|^2 {\rm d} \lambda(x,y).
    \end{equation*}
    The (squared) \emph{Wasserstein distance} between $\mu_1, \mu_2$ is the minimal cost of a transport plan:
    \begin{equation*}
        \Wd(\mu_1,\mu_2) = \min_{\lambda \in \Pi(\mu_1, \mu_2)} \cost(\lambda).
    \end{equation*}
\end{definition}

\begin{remark}
    If $\mu_1, \mu_2$ are finitely supported probability measures, then every transport plan from $\mu_1$ to $\mu_2$ will have support contained in $\supp \mu_1 \times \supp \mu_2$. In particular, they will also have finite support. In this setting we do not have to deal with most measure theoretic issues. See for example \cite[Chapter 7]{Villani03:TopicsOptTrans} for the general definition.
\end{remark}

For measures associated to polynomials, it will be convenient to work with doubly stochastic matrices.
For $d\in \N$, we let $B_d$ denote the Birkhoff polytope of doubly stochastic matrices of size $d \times d$. Recall that a matrix $M \in \R_{\geq 0}^{d\times d}$ is called \emph{doubly stochastic} if all rows and all columns add up to $1$.

\begin{lemma}
\label{lem:birkhoff_transport}
Let $p, q \in \C[z]$ be polynomials of degree $d$ with roots $\alpha_1, \dots, \alpha_d$ and $\beta_1, \dots, \beta_d$ respectively. Suppose that $p$ and $q$ have no multiple roots. Then the following is a bijection:
\begin{align*}
    \lambda : B_d &\rightarrow \Pi(\mu(p),\mu(q)) \\
    M &\mapsto \lambda(M) \coloneqq \frac{1}{d} \sum_{1\leq i,j \leq d} M_{ij} \delta_{\alpha_i, \beta_j}.
\end{align*}
\end{lemma}
\begin{proof}
    Since we assume $p$ and $q$ have no multiple roots, it follows that the map
    \begin{equation*}
        M \mapsto \frac{1}{d} \sum_{1\leq i,j \leq d} M_{ij} \delta_{\alpha_i, \beta_j}
    \end{equation*}
    defines a bijection between $\R^{d\times d}_{\geq 0}$ and the set of \emph{finite} measures supported on the points
    $\{\alpha_1, \dots, \alpha_d\} \times \{\beta_1, \dots, \beta_d\}$.
    For $M \in \R^{d\times d}_{\geq 0}$ and $1\leq i \leq d$, we have
    \begin{equation*}
        \pi_{1*}(\lambda(M))[\{\alpha_i\}] =
        (\lambda(M))[\{\alpha_i\} \times \C] =
        \sum_{1\leq j \leq d} (\lambda(M))[\{(\alpha_i, \beta_j)\}]
        = \frac{1}{d} \sum_{1\leq j \leq d} M_{ij}. 
    \end{equation*}
    From this we see that $\pi_{1*}(\lambda(M)) = \mu(p)$ if and only if all row sums of $M$ are equal to $1$. Similarly $\pi_{2*}(\lambda(M)) = \mu(q)$ if and only if all column sums are equal to $1$. So $\lambda(M)$ is a transport plan if and only if $M \in B_d$.
\end{proof}
\begin{remark}
If $p$ and $q$ are allowed to have multiple roots, then $\Pi(\mu(p), \mu(q))$ is given by a transportation polytope (see \cite{DeLoeraKim14} for an overview). However,
the map defined in \Cref{lem:birkhoff_transport} is still a surjection if we list the roots with multiplicity.
\end{remark}

We use \Cref{lem:birkhoff_transport} in our consideration of the Wasserstein distance for polynomials.
From now on, we consider univariate polynomials $p,q \in \C[z]$ with factorizations
$p(z) = \prod_{i=1}^d (z - \alpha_i)$, $q(z) = \prod_{i=1}^d (z - \beta_i)$, without multiple roots. For $M \in B_d$, we write
\begin{equation}\label{eqn:cost_M}
    \cost(M, p, q) = \cost(\lambda(M)) = \frac{1}{d} \sum_{1\leq i,j \leq d} M_{ij} \| \alpha_i -  \beta_j \|^2.
\end{equation}
By these considerations and using the bijection in \Cref{lem:birkhoff_transport} we interpret the Wasserstein distance between two polynomials as an optimization problem:
\begin{equation*}
    \Wd(p,q) \coloneqq \Wd(\mu(p), \mu(q)) = \min_{M \in B_d} \cost(M,p,q).
\end{equation*}

\section{Invariant Birkhoff polytope}\label{sec:Birkhoff}

We now restrict to polynomials with real coefficients and study the effects of complex conjugation, denoted as usual by $\bar{z}$. If $p,q \in \R[z]$, then complex conjugation will induce an involution on the transport plans form $\mu(p)$ to $\mu(q)$. We compute this involution for the Birkhoff polytope under the identification of \Cref{lem:birkhoff_transport}.

Assume $p, q \in \R[z]$ with roots $\alpha_i$ and $\beta_i$ as above. We denote by $\phi$ and $\psi$ the permutations that swap complex conjugate pairs of roots of $p$ and $q$ respectively, namely,
\begin{equation}\label{eq:phipsi_complconj}
    \begin{aligned}
    \alpha_{\phi(i)} &= \overline{\alpha_{i}}, \\
    \beta_{\psi(i)} &= \overline{\beta{i}}.
    \end{aligned}
\end{equation}
We will call $P_\phi, P_\psi$ the permutation matrices associated to $\phi$ and $\psi$ respectively. Then the pair $(\phi,\psi)$ defines a map $\iota : B_d \rightarrow B_d$ via
\[
(\iota M)_{ij} = M_{\phi(i)\psi(j)} = (P_\phi M P_\psi)_{ij}. 
\]
We can see that $\iota$ is a linear isometry of the Birkhoff polytope. In fact, every combinatorial symmetry of $B_d$ is of the form $M \mapsto P_g M P_h$ or $M \mapsto P_g M^\top P_h$ for some permutations $g,h \in \Sym(d)$, see \cite{BaumeisterLadisch18}.
In particular, $\iota$ defines an involution on the Birkhoff polytope. This suggests the following definition.
\begin{definition}\label{def:invBirkhoff}
    Let $p, q \in \R[z]$ be polynomials of degree $d$ with distinct roots and take $\phi,\psi \in \Sym(d)$ to be permutations that satisfy \eqref{eq:phipsi_complconj}. Consider the operation $\iota = (\phi,\psi)$. The associated \textit{$\iota$-invariant Birkhoff polytope}, denoted $\invp$, is the set of all doubly stochastic matrices that are invariant under $\iota$.
\end{definition}

Observe that the cost function does not change under $\iota$, namely
\[
\cost(M,p,q) = \cost(\iota M, p,q) = \cost\left(\frac{M + \iota M}{2}, p,q \right).
\]
Therefore, instead of solving an optimization problem over the Birkhoff polytope, we can find the (squared) Wasserstein distance using the $\iota$-invariant Birkhoff polytope:
\[
\Wd(p,q) = \frac{1}{d} \min_{M \in \invp} \cost(M,p,q).
\]
The cost function remains linear for $M \in \invp$, hence the minimum is attained at a vertex of $\invp$. The $\iota$-invariant Birkhoff polytope can reduce the complexity of the optimization problem, since it is potentially lower dimensional.
We give the formula for the dimension of $\invp$ in terms of the number of real and complex roots of the polynomials $p,q$. 

Let us introduce some notation. We denote the number of real roots of a polynomial $p \in \R[z]$ by $p_{\R}$, and the number of \emph{pairs} of its complex conjugate roots by $p_{\C}$.

\begin{lemma}\label{lem:dimBiota}
Suppose $\iota=(\phi,\psi)$ is the involution associated with $p,q$. Then, the dimension of the $\iota$-invariant Birkhoff polytope $B_d^{\iota}$ is
$$\dim B_d^{\iota}=(\pr+\pc)(\qr+\qc)-(\pr+\pc+\qr+\qc-1)+\pc\qc.$$
\end{lemma}

\begin{proof}
By definition, the linear span of $B_d^{\iota}$ coincides with the linear space
\[
V:=\left\lbrace M\in \mathbb{R}^{d\times d}: M_{ij}=M_{\phi(i)\psi(j)}, \sum_{i=1}^d M_{ij}=1, \forall j\in [d],\, \sum_{j=1}^d M_{ij}=1, \forall i\in [d] \right\rbrace.
\]
Using coordinates $(M_{ij})_{i,j=1}^d \in \R^{d\times d}$, on the affine subspace $V$ the following relations hold: $\sum_{j=1}^dM_{ij}\!=\!\sum_{j=1}^d M_{\phi(i)j}$ for $i\!=\!1,\ldots,d$, and $\sum_{i=1}^d M_{ij}\!=\!\sum_{i=1}^d M_{i\psi(j)}$ for $j\!=\!1,\ldots,d$.
Among the relations that row-sums and column-sums of $M$ are $1$, there are exactly $\pr+\pc+\qr+\qc$ distinct linear relations but only $\pr+\pc+\qr+\qc-1$ of them are independent.

Consider a new matrix $N$ constructed from $M$ as follows: for every pair $(i,\phi(i))$ with $i < \phi(i)$, delete the $\phi(i)$-th row of $M$; for every pair $(j,\psi(j))$ with $j < \psi(j)$, delete the $\psi(j)$-th column of $M$; for the entry $M_{ij}$ that remains, if $i\neq\phi(i)$ and $j\neq\psi(j)$, we replace $M_{ij}$ by $M_{ij}+M_{i\psi(j)}.$
For the rows and the columns of $N$ we use the indices inherited from $M$. 
Each of the $(\pr+\pc)(\qr+\qc)$ entries of $N$ is a sum of 1 or 2 of the  $(\pr+\pc)(\qr+\qc)+\pc\qc$ entries of $M$. No entry of $M$ appears in more than one entry of $N$. We denote the linear space spanned by the $(\pr+\pc)(\qr+\qc)+\pc\qc$ entries of $M$ by $W$.

Suppose that $i$ is a valid row index for $N$, then the $i$-th row-sum of $M$ is either $\sum_j N_{ij}$ if $i\neq\phi(i)$, or $\sum_j a_j N_{ij}$, where $a_j=2$ if $j\neq\psi(j)$ and $a_j=1$ otherwise. Similar formulae hold for column-sums of $M$. 
Suppose $m=\pr+\pc+\qr+\qc-1$ and $\ell_1,\ldots,\ell_m$ are the linear relations that the weighted row-sums and column-sums of $N$ equal to $1$, except for the weighted sum of the last row of $N$. We assume that the first $\qr+\qc$ linear relations correspond to weighted column-sums of $N$. These linear relations correspond to a basis of linear relations that the row and column sums of $M$ are equal to 1.
Suppose $\sum_{i=1}^m \lambda_i \ell_i=0$ for some $\lambda_i\in \mathbb{R}$. For $i=1,\ldots,\qr+\qc$, each $\ell_i$ contains the $i$-th entry of the last row of $N$ that does not appear in any other $\ell_j$ with $j\neq i$. This implies that the corresponding scalars $\lambda_i$ are zero. We are then left with $\sum_{i=\qr+\qc}^m \lambda_i \ell_i=0$, which gives the trivial result $\lambda_i = 0$ for all $i$, because the conditions on the weighted row-sums of the first $\pr+\pc-1$ rows of $N$ are linearly independent.
Hence, $\ell_1,\ldots,\ell_{m}$ are linearly independent over the linear span of the entries of $N$, and over $W$ as well. We conclude 
\[
\dim B_d^{\iota}=\dim V=(\pr+\pc)(\qr+\qc)+\pc\qc-(\pr+\pc+\qr+\qc-1). \qedhere
\]
\end{proof}

In the following subsections, we will focus on the vertices of this polytope. In order to get a better understanding, we will introduce a graph associated to the vertices of $\invp$. To conclude this part, we construct some examples of $\iota$-invariant Birkhoff polytopes, which we will review later using the graph interpretation.

\begin{example}
    Let $d=2$, then $p,q$ are quadratic polynomials with real coefficients. The Birkhoff polytope $B_2$ is a segment in $\R^4$. There are three possible cases for the structure of the roots of $p,q$: both polynomials have both roots real; one has both roots real and one has both roots complex; both polynomials have both roots complex. 
    In the first case, we have $P_\phi = P_\psi = \1$, hence $\invptwo = B_2$. 
    In the second case, we assume that $p$ has complex roots and $q$ real roots, then $P_\phi = \left( \begin{smallmatrix} 0 & 1 \\ 1 & 0 \end{smallmatrix} \right)$, $P_\psi = \1$, and the polytope becomes a point: $\invptwo = \left\lbrace \frac{1}{2} \left( \begin{smallmatrix} 1 & 1 \\ 1 & 1 \end{smallmatrix} \right)\right\rbrace$.
    Finally, in the case where the roots of both polynomials are complex, $P_\phi = P_\psi = \left( \begin{smallmatrix} 0 & 1 \\ 1 & 0 \end{smallmatrix} \right)$ hence $\iota$ fixes the two vertices of the Birkhoff polytope, so $\invptwo = B_2$.
\end{example}

\begin{example}\label{ex:invBthree}
    Let $d=3$, then $p,q$ are cubic polynomials with real coefficients. The polytope $B_3$ has dimension $4$ in $\R^9$. There are again three cases of roots configurations for $p,q$:
    \begin{enumerate}
        \item $p,q$ have all roots real, then $P_\phi = P_\psi = \1$, so $\invpthree = B_3$.
        \item $q$ has all roots real and $p$ has exactly one real root. We will assume that the real root of $p$ is $\alpha_1$. Then $P_\phi = \left( \begin{smallmatrix} 1 & 0 & 0 \\ 0 & 0 & 1 \\ 0 & 1 & 0 \end{smallmatrix} \right)$, $P_\psi = \1$. The $6$ vertices of the Birkhoff polytope collapse to form a triangle 
        \[
        \invpthree = \operatorname{conv} \left\lbrace \left( \begin{smallmatrix} 1 & 0 & 0\\ 0 & \frac{1}{2} & \frac{1}{2}\\ 0 & \frac{1}{2} & \frac{1}{2} \end{smallmatrix} \right), \left( \begin{smallmatrix} 0 & \frac{1}{2} & \frac{1}{2}\\ 1 & 0 & 0\\ 0 & \frac{1}{2} & \frac{1}{2} \end{smallmatrix} \right), \left( \begin{smallmatrix} 0 & \frac{1}{2} & \frac{1}{2}\\ 0 & \frac{1}{2} & \frac{1}{2}\\ 1 & 0 & 0 \end{smallmatrix} \right) \right\rbrace.
        \]
        \item $p, q$ both have exactly one real root. We will assume that these are $\alpha_1, \beta_1$. Then $P_\phi = P_\psi = \left( \begin{smallmatrix} 1 & 0 & 0 \\ 0 & 0 & 1 \\ 0 & 1 & 0 \end{smallmatrix} \right)$. The associated invariant Birkhoff polytope is now the square
        \[
        \invpthree = \operatorname{conv} \left\lbrace \mathds{1}, \left( \begin{smallmatrix} 1 & 0 & 0 \\ 0 & 0 & 1\\ 0 & 1 & 0 \end{smallmatrix} \right),
        \left( \begin{smallmatrix} 0 & \frac{1}{2} & \frac{1}{2}\\ \frac{1}{2} & \frac{1}{2} & 0 \\ \frac{1}{2} & 0 & \frac{1}{2} \end{smallmatrix} \right), 
        \left( \begin{smallmatrix} 0 & \frac{1}{2} & \frac{1}{2} \\ \frac{1}{2} & 0 & \frac{1}{2}\\ \frac{1}{2} & \frac{1}{2} & 0 \end{smallmatrix} \right)\right\rbrace.
        \]
    \end{enumerate}
\end{example}

\begin{example}
    Let $d=4$, then $p,q$ are quartic polynomials with real coefficients. The Birkhoff polytope $B_4$ has dimension $9$ in $\R^{16}$. Assume that $p$ has only real roots, and $q$ has only complex roots. Then, as predicted in \Cref{lem:dimBiota}, the invariant Birkhoff polytope is a $3$-dimensional octahedron. This and all the other cases are listed in \Cref{table:polytope dimension}.
    \begin{table}[ht]
        \centering
        {\def\arraystretch{1.5}
        \begin{tabular}{>{\centering\arraybackslash}m{2cm} | >{\centering\arraybackslash}m{2cm} |>{\centering\arraybackslash}m{2cm}|>{\centering\arraybackslash}m{3cm}  }
        $(p_\R,p_\C)$ & $(q_\R,q_\C)$ & $\dim B_d^{\iota}$ & $|\operatorname{vertices}(B_d^{\iota})|$\\
        \hline
        \hline
        (4,0) & (4,0) & 9 &24\\
        (4,0) & (2,1) & 6 &12\\
        (4,0) & (0,2) & 3 &6\\
        (2,1) & (2,1) & 5 &13\\
        (2,1) & (0,2) & 4 &12\\
        (0,2) & (0,2) & 5 &8\\
        \end{tabular}
        }
        \caption{Dimensions of the invariant Birkhoff polytopes for $d=4$.}
        \label{table:polytope dimension} 
    \end{table}
\end{example}

\subsection{Graph interpretation}\label{subsec:graph}
The aim of this subsection is to introduce a graph interpretation of the vertices of $\invp$. This interpretation will be used to realize all the vertices of $\invp$ as minimizer of our optimization problem \eqref{eq:opt_problem} and we shall see that the graph automorphisms are closely related to the Wasserstein distance degree.

Recall that the $\iota$-invariant Birkhoff polytope $\invp$ is the image of the standard Birkhoff polytope via the map $v \mapsto \frac{v+\iota(v)}{2}$. Any vertex of the $\iota$-invariant Birkhoff polytope $\invp$ is necessarily the image of a vertex of the Birkhoff polytope. Otherwise, it would be possible to express it as the convex combination of other points in $\invp$. 
It is a classical result due to Birkhoff \cite{Birkhoff1946} and von Neumann \cite{Neumann1953} that the vertices of the Birkhoff polytope are permutation matrices.
As a consequence, suppose $M\in \R^{d\times d}$ be a vertex of $\invp$, then $M_{ij}\in \left \{0,\frac{1}{2},1\right\}$. Hence, the matrix $2M$ has integer entries.

We associate a bipartite graph $\Gamma$ with $d$ nodes on the left side (L) and $d$ nodes on the right side (R), such that $2M$ is its adjacency matrix. Namely, the vertices of the graph, denoted $V(\Gamma) = [d] \times \{L,R\}$, satisfy the property that $(i,L)$ and $(j,R)$ are connected by $2M_{ij}$ edges. Such an edge will be denoted by $(i,j)$, where the first entry indicates a vertex $(i,L)$ and the second entry indicates a vertex $(j,R)$. Since $M$ is a doubly stochastic matrix, every vertex in $\Gamma$ has degree $2$. $\Gamma$ is then a disjoint union of cycles, and because it is bipartite, each cycle has even length.

The invariance of $M$ under $\iota$ gives a graph automorphism that we will by abuse of notation also call $\iota$. Explicitly, $\iota(i,L) = (\phi(i),L)$ and $\iota(j,R) = (\psi(j),R)$.
Note that $\iota$ is also an involution and it acts independently on both sides of the bipartite graph $\Gamma$. 

Let us introduce some types of cycles in $\Gamma$, governed by the involution:
\begin{enumerate}
    \item $2$-cycles between fixed points of $\iota$;
    \item pairs of $2$-cycles that are interchanged by $\iota$;
    \item cycles of length $2(2k +1),\, k > 0$, including exactly one fixed point of $\iota$ on each side;
    \item cycles of length $4k,\, k > 0$, including exactly two fixed points of $\iota$ on one side and none on the other.
\end{enumerate}
The following result states that these are the only possible disjoint cycles in our graphs.

\begin{theorem}\label{thm: vertices of gamma}
Let $M$ be a vertex of $\invp$ and $\Gamma$ be the associated graph. Then, the disjoint cycles in $\Gamma$ are of type $1$--$4$.
Conversely, let $\Gamma$ be a bipartite graph with $d$ vertices on each side, such that all its vertices have degree $2$. If all the disjoint cycles of $\Gamma$ are of type $1$--$4$, then the graph corresponds to a vertex $M\in \invp$.
\end{theorem}

\begin{proof}
Suppose that $M$ is a vertex of $\invp$ corresponding to the graph $\Gamma$.
Since $M$ is the image of a permutation matrix via $\frac{\cdot \,+\, \iota(\cdot)}{2}$, the entry of $M$ corresponding to a pair of fixed points of $\iota$ is either $0$ or $1$. Hence, if there is an edge in $\Gamma$ between two fixed points of $\iota$, then there must be two such edges. The latter case describes cycles $C$ in $\Gamma$ of type $1$.

If a cycle $C$ contains no fixed point of $\iota$, then $C$ must have length $2$. 
Otherwise, we could decompose $C$ (and $\iota(C)$ if $C$ is not set-wise fixed) into two distinct $\iota$-invariant perfect matchings. This translates into an expression for $M$ as the midpoint of two other points in $\invp$, contradicting the vertex assumption. 
Such a cycle of length $2$ must therefore be (one cycle of the pair) of type $2$.

Next, suppose that $C$ is a cycle of length at least $4$. Then, it must contain a fixed point of $\iota$. 
Since $\iota$ is an involution on $C$, it either contains $2$ fixed points, or all the points in $C$ are fixed. We already observed that between two fixed points there cannot be simple edges, so the second option leads to a contradiction. Hence, $C$ must have exactly $2$ fixed points and they can either lie on the same side (cycle of type 4) or one on each side (cycle of type 3).

Conversely, let $M$ be the matrix corresponding to the bipartite graph $\Gamma$ with cycles of types $1$--$4$. Then, $M$ is a doubly stochastic matrix and it is invariant under $\iota$ so $M\in \invp$. We prove that $M$ is a vertex. 
Assume that $M=\frac{N+L}{2}$ for some $N,L\in \invp$. Note that $N_{ij}=L_{ij}=0$ whenever $M_{ij}=0$, and $N_{ij}=L_{ij}=1$ whenever $M_{ij}=1$. This is the case for $(i,j)$ with the edge between $(i,L), (j,R)$ in a cycle of type $1$ or $2$. Let $C$ in $\Gamma$ be a cycle of type $3$ or $4$. 
Let $(i_1,j_1),(i_2,j_2),\ldots,(i_\ell,j_\ell)$ be the indices of entries of $M$ corresponding to edges in $C$. 
Without loss of generality, assume $i_1=i_2=\phi(i_1)$ and $j_1=\psi(j_2)$. The only two positions in the $i_1$ row filled with nonzero entries in $N$ or $L$ are $(i_1,j_1)$ and $(i_1,j_2)$.
Let $L_{i_1,j_1}=\epsilon$, then $L_{i_1,j_2}=1-\epsilon$. Since $L$ is invariant under $\iota$, we must have $\epsilon=1-\epsilon=\frac{1}{2}$. Similarly, $N_{i_1,j_1}=N_{i_1,j_2}=\frac{1}{2}$. 
Analogously, we have $N_{i_k,j_k}=L_{i_k,j_k}=M_{i_k,j_k}=\frac{1}{2}$ for all $k=1,\ldots,\ell$.
Hence, $L=N=M$ and $M$ is a vertex. 
\end{proof}

We are interested in the cardinality of the group automorphisms of the graph $\Gamma$ that commute with $\iota$ and preserve both sides. Such number relates to the algebraic degree of the Wasserstein distance, as it will be explained in \Cref{sec:alg_deg}.
The formula for the cardinality relies on the number of cycles of each type. We denote the number of cycles of type $1$ and of pairs of cycles of type $2$ by $c_{1}$ and $c_{2}$ respectively; the number of cycles in the third class with length $2(2k+1)$ is $c_{3,k}$ and the number of cycles in the fourth class with length $4k$ is $c_{4,k,L}$ or $c_{4,k,R}$, depending on the side of the fixed points.

\begin{corollary}
Let $\Gamma$ be the bipartite graph corresponding to a vertex $M$ in $\invp$ and let $S$ be the set of disjoint cycles (or type 2 cycle pairs) in $\Gamma$.
The group of automorphisms $\Aut_\iota(\Gamma)$ of $\Gamma$ that commute with $\iota$ and preserve both sides has cardinality
    \begin{align*}
        |\Aut_\iota(\Gamma)| &= \prod_{C \in S} |\Aut_\iota(C)|c_{1}!\, c_{2}! \prod_k c_{3,k}!\, c_{4,k,L}!\, c_{4,k,R}! \\
        &=2^{c_{2}+\sum_k (c_{3,k}+2c_{4,k,L}+2c_{4,k,R})}  c_{1}!\, c_{2}! \prod_k ( c_{3,k}!\, c_{4,k,L}!\, c_{4,k,R}! ).
    \end{align*}
\end{corollary}
\begin{proof}
The automorphisms in $\Aut_\iota(\Gamma)$ respect the classification of disjoint cycles in $\Gamma$. We can freely permute inside these classifications, hence obtain the factorial terms in the formula. We are left to compute the automorphisms that commute with $\iota$ inside each disjoint cycle of the corresponding type:
 \begin{enumerate}
    \item trivial symmetry
    \item two automorphisms: $\iota$ and the trivial symmetry.
    \item two automorphisms: $\iota$ and the trivial symmetry.
    \item Kleinean group of symmetry, namely we are allowed to swap the two fixed points and also act by $\iota$.\qedhere
\end{enumerate}
\end{proof}

\begin{example}\label{ex:invBthree_graph}
    Let $d = 3$ and assume that $p,q$ both have exactly $1$ real root. Then, from case $3$ in \Cref{ex:invBthree} we know that $\invpthree$ is a square. Its representation in terms of the associated graphs is shown in \Cref{fig:square_deg3}. The graph $\Gamma$ corresponding to either of the bottom vertices consists of one cycle of type 3 with length $6$. The graph automorphisms $\Aut_\iota(\Gamma)$ are just the identity and $\iota$ itself. On the other hand, the graph $\Gamma$ corresponding to either of the top vertices in \Cref{fig:square_deg3} consists of one cycle of type 1 and one pair of cycles of type 2. The graph automorphisms $\Aut_\iota(\Gamma)$ also consist of the identity and $\iota$ itself.
    \begin{figure}[ht]
        \centering
        \includegraphics[width = 0.36\textwidth]{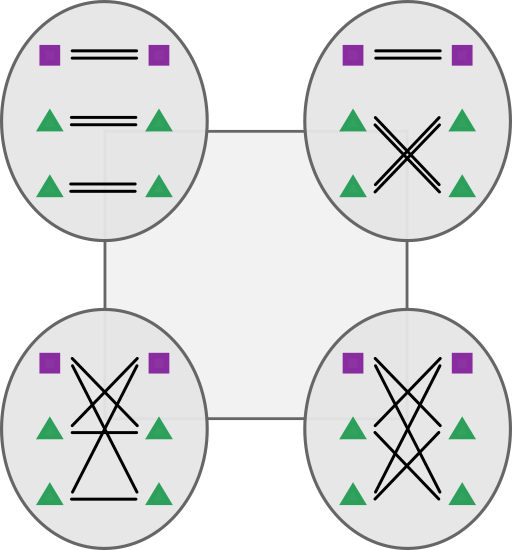}
        \caption{The invariant square $\invpthree$ from \Cref{ex:invBthree_graph}. Its vertices correspond to bipartite graphs $\Gamma$. The purple squares represent real roots, the green triangles represent complex roots. The involution $\iota$ exchanges pairs of green triangles.}
        \label{fig:square_deg3}
    \end{figure}
\end{example}

\subsection{Realization of the vertices}\label{subsec:vertices}
Any vertex of the $\iota$-invariant Birkhoff polytope can be a minimizer for the cost funciton $\cost(M, p, q)= \frac{1}{d} \sum_{1\leq i,j \leq d} M_{ij} \| \alpha_i -  \beta_j \|^2$ for $(p,q)$ in some non-empty open set of monic polynomials without multiple roots. Note that we always consider the roots of the polynomials $p,q$ as ordered. This is implicitly assumed also for the open sets of monic polynomials, since the ordering of the roots of $p$ (resp. $q$) induces a coherent ordering of the roots of polynomials in a sufficiently small neighborhood of $p$ (resp. $q$). In this open neighborhood the involutions $\phi, \psi$ induced by complex conjugation remain the same. 

\begin{proposition}
\label{prop:realization}
Let $\phi,\psi$ be involutions of $\Sym(d)$, $\iota=(\phi,\psi)$, and let $M \in B^\iota_d$ be a vertex.
The set of $p,q \in \R[z]$ that have no multiple roots, factorize as $p(z) = \prod_{i=1}^d (z - \alpha_i)$,
$q(z)=\prod_{j=1}^d (z - \beta_j)$,  with $\overline{\alpha_i} = \alpha_{\phi(i)}, \overline{\beta_i} = \beta_{\psi(i)}$, and for which $M$ is the unique minimizer for $\cost(M,p,q)$ in $B^\iota_d$, is an open non-empty set denoted by $U_{(\phi,\psi)}^M$.
\end{proposition}
\begin{proof}
The roots of a polynomial vary continuously as its coefficients vary. So, if $M$ is the unique minimizer for the cost function of two polynomials $p,q$ with roots ordered corresponding to the involution $\phi,\psi$, then there is an open set of polynomial pairs around $(p,q)$ such that any polynomial in the pair has distinct roots (ordered to satisfy the involutions $\phi,\psi$), and $M$ is the unique minimizer of the cost function for all polynomial pairs in the open set. The unique minimizer requirement is an open condition, since it means that $\cost(M,p,q)<\cost(M',p,q)$ for any other vertex $M'\in \invp$.

We are left to construct the roots of a pair of polynomials of degree $d$ such that $M$ is the unique minimizer for the pair, to prove that $U_{(\phi,\psi)}^M$ is non-empty. Let $\Gamma$ be the graph associated to $M$. Recall that $\Gamma$ consists of disjoint cycles of four different types with cardinalities $c_1,c_2,c_{3,k},c_{4,k,L},c_{4,k,R}$. Denote the cycles by $C_1,\ldots,C_\ell$. Let $\omega=e^{\frac{\pi i}{3}}$; we will do our construction on the Eisenstein lattice $\mathbb{Z}[1,\omega]$. The idea is to repeat our construction for each disjoint cycle (or pairs of disjoint cycles, if they are of type $2$) and make sure that distinct disjoint cycles are far away from each other.
\begin{itemize}
\item $C_k$ is of type $1$: assign the value $2dk$ and $2dk\!+\!1$ to the corresponding real roots of $p$ and $q$ respectively.
\item $C_k$ is of type $2$: assign the values $2dk\!+\!\omega$, $2dk\!+\!\overline{\omega}$ to the pair of roots of $p$, and $2dk\!+\!\omega\!+\!1$, $2dk\!+\!\overline{\omega}\!+\!1$ to the pair of roots of $q$. In this way, the vertex in $\Gamma$ corresponding to $2dk\!+\!\omega$ connects to the one of $2dk\!+\!\omega\!+\!1$, and analogously  the vertex in $\Gamma$ corresponding to $2dk\!+\!\overline{\omega}$ connects to the one of $2dk\!+\!\overline{\omega}\!+\!1$.
\item $C_k$ is of type $3$: assume its length is $2(2j\!+\!1)$. Assign the values 
$2dk$, $2dk\!+\!\omega$, $2dk\!+\!\omega\!+\!1$, $\ldots,$ $2dk\!+\!\omega\!+\!2j\!-\!1$, $2dk\!+\!2j$, $2dk\!+\!\overline{\omega}\!+\!2j\!-\!1$, $\ldots,$ $2dk\!+\!\overline{\omega}$ to the roots of $p,q$ alternated.
\item $C_k$ is of type $4$: assume its length is $4j$ and that the fixed points are on the left side. Assign the values $2dk$, $2dk\!+\!\omega$, $2dk\!+\!\omega\!+\!1$, $\ldots$, $2dk\!+\!\omega\!+\!2j\!-\!2$, $2dk\!+\!2j\!-\!1$, $2dk\!+\!\overline{\omega}\!+\!2j\!-\!2$, $\ldots,$ $2dk\!+\!\overline{\omega}$ to the roots of $p,q$ alternated.
\end{itemize}
See 
\Cref{fig:eisenstein_realization} for an illustration.
\begin{figure}[ht]
    \centering
\includegraphics[width=\textwidth]{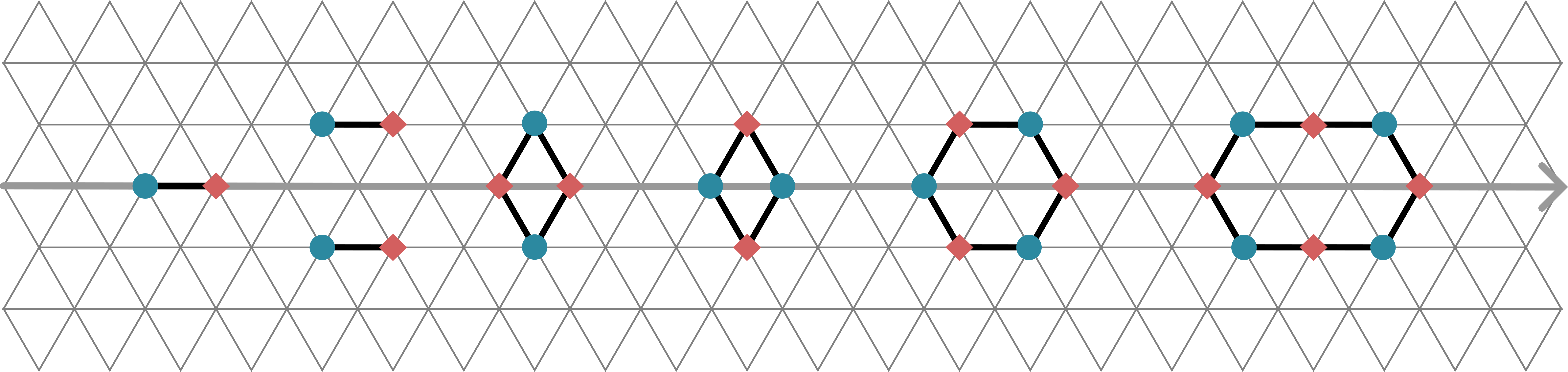}
    \caption{Disjoint cycles in $\Gamma$ as subgraphs of the Eisenstein lattice. From left to right, we have disjoint cycles of types 1, 2 (pair), 4, 4, 3, 4.}
    \label{fig:eisenstein_realization}
\end{figure}
By the above construction, for any pair of $(i,j)$ such that $M_{i,j}\neq 0$, we have $\| \alpha_i -  \beta_j \|=1$. Any two distinct lattice points in the Eisenstein lattice have distance at least $1$, so $M$ is a minimizer for the cost function of $p,q$. 
Suppose there exists another vertex $N\in B_d^{\iota}$ with $\cost(M)~=~\cost(N)$ and denote the graph corresponding to $N$ by $\Gamma^\prime$. Then, $\| \alpha_i -  \beta_j \|=1$ whenever $N_{ij}\neq 0$. Since any two lattice points in disjoint cycles (or pairs, for type $2$) have at least distance $d$, the disjoint cycles of $\Gamma^\prime$ have the same vertices of the disjoint cycles of $\Gamma$. Inside each disjoint cycle (or pairs, for type $2$), the only way to achieve $\| \alpha_i -  \beta_j \|=1$ whenever $N_{ij}\neq 0$ is to assign the roots to $p,q$ exactly as for $M$. Hence, $M=N$ and $M$ is the unique minimizer for the cost function of $p,q$.
\end{proof}

\begin{remark}
    Note that the set $U_{(\phi,\psi)}^M$ is also semialgebraic, since it is a projection of an appropriate semialgebraic sets arising from the conditions that the pairs of factorized polynomials have a coherent enumeration of their roots, and from $M$ being the unique minimizer.
    More explicitly, these conditions are polynomial discriminants and inequalities of type $\cost(M,p,q)\leq\cost(M',p,q)$.
\end{remark}

\section{Algebraic degree formulae}\label{sec:alg_deg}

Based on the combinatorial data of the previous section, here we construct a formal candidate of the minimal polynomial for the (squared) Wasserstein distance. The actual minimal polynomial will always divide the formal candidate. We start by reviewing useful tools from Galois theory.

Consider the ring $ A = \Q[x_1,\dots,x_d,y_1,\dots y_d]$. We denote by $\cdot$ the following group action of
$G = \Sym(d) \times \Sym(d)$ on A:
\begin{align*}
  (g_1,g_2) \cdot x_i &= x_{g_1(i)}, \\
  (g_1,g_2) \cdot y_j &= y_{g_2(j)}.
\end{align*}
Let $\sigma_1, \dots, \sigma_d$ be the elementary symmetric polynomials in the variables $x_1, \dots x_d$, and $\tau_1, \dots, \tau_d$ be the elementary symmetric polynomials in $y_1, \dots y_d$.
Then the following holds.
\begin{lemma}
  The ring of invariants $A^G$ is the polynomial ring $\Q[\sigma_1, \dots, \sigma_d,\tau_1, \dots \tau_d]$. In particular, the generators $\sigma_1, \dots, \sigma_d,\tau_1, \dots \tau_d$ are algebraically independent.
\end{lemma}
\begin{proof}
    Let us consider the subgroup $H = \Sym(d) \times 1$ of $G$, so $H$ permutes the coordinates $x_i$. By the fundamental theorem of symmetric polynomials (over the ring $\Q[y_1,\dots, y_d]$), we have $A^H = \Q[\sigma_1, \dots \sigma_d, y_1, \dots y_d]$ with algebraically independent generators $\sigma_i, y_j$. Such $H$ is normal in $G$, and $G/H$ acts on $A^H$ by permuting the coordinates $y_i$.
    Using the fundamental theorem of symmetric polynomials over the ring $\Q[\sigma_1, \dots, \sigma_d]$, we obtain
  \[
  A^G = (A^H)^{(G/H)} = \Q[\sigma_1, \dots \sigma_d, y_1, \dots y_d]^{\Sym(d)} =
  \Q[\sigma_1, \dots \sigma_d,\tau_1, \dots \tau_d]
  \] 
  with algebraically independent generators. 
\end{proof}
From this, we get that the inclusion $\Q(\sigma_1, \dots \sigma_d,\tau_1, \dots \tau_d) \subset \Q(x_1,\dots,x_d,y_1,\dots y_d)$ is a Galois field extension with Galois group $G$.
\begin{lemma}
\label{lem:galois_conjugates}
  Let $f \in A$ and let $\{f_1 = f, f_2, \dots, f_k\}$ be the orbit of $f$ under $G$. 
  Consider the following product in $A[t]$:
  \begin{equation}
    h \coloneqq \prod^k_{i=1} (t - f_i).
  \end{equation}
  Then, $h$ actually lies in $A^G[t]$ and it is an irreducible polynomial in $\Q(x_1,\dots,x_d,y_1,\dots y_d)[t]$ of degree $k$ in $t$.
\end{lemma}
\begin{proof}
    This is a classical exercise in Galois theory about Galois conjugates, see \cite[p. 573]{DummitFoote04:AbstractAlgebra}.
    We give a proof sketch for convenience.
    The action of $G$ permutes the factors of $h$, so $h$ lies in $(A[t])^G = A^G[t]$.
    Suppose $h= h_1 h_2$ is a factorization of $h$ in $\Q(x_1,\dots,x_d,y_1,\dots y_d)$, and assume without loss of generality that $(t - f_1)$ divides $h_1$ in $\Q(x_1,\dots,x_d,y_1,\dots,y_d)[t]$.
    Applying the Galois action, we get that every linear factor $(t - f_i)$ divides $h_1$. Since the $f_i$ are pairwisely distinct, we get that $h_1$ must have degree $k$ in $t$, so the factorization is trivial.  
\end{proof}

\subsection{Algebraic degrees}

Given two involutions $\phi, \psi \in \Sym(d)$ and a vertex $M$ of $B^\iota_d$, consider the polynomial 
\begin{equation}\label{eqn:f_M}
  f_M  \coloneqq \frac{1}{d} \sum_{1\leq i,j \leq d} M_{ij} (x_i - y_j) (x_{\phi(i)} - y_{\psi(j)})
  \subset \Q[x_1,\dots,x_d,y_1,\dots y_d].
\end{equation}
This polynomial is relevant for us, because of the following observation.

\begin{lemma}
    \label{lem:specialize}
    Consider $p, q \in \R[z]$ of degree $d$ with factorizations
    $p(z) = \prod_{i=1}^d (z - \alpha_i)$, $q(z) = \prod_{i=1}^d (z - \beta_i)$, so that 
    $\overline{\alpha_i} = \alpha_{\phi(i)}, \overline{\beta_i} = \beta_{\psi(i)}$.
    Let $\Phi_{p,q} : \Q[x_1, \dots, x_d, y_1, \dots y_d] \rightarrow \C$ be the map given by
    $x_i \mapsto \alpha_i, y_i \mapsto \beta_i$. Then, for a vertex $M \in \invp$, we have
    \begin{equation*}
            \Phi_{p,q}(f_M) = \cost(M,p,q).
    \end{equation*}
\end{lemma}
\begin{proof}
    This follows directly from $\|\alpha_i - \beta_j\|^2 = (\alpha_i - \beta_j) \overline{(\alpha_i - \beta_j) } =  (\alpha_i - \beta_j)(\alpha_{\phi(i)} - \beta_{\psi(j)})$ and comparing \cref{eqn:cost_M,eqn:f_M}.
\end{proof}
We also have an action of $G$ on $[d] \times \{L, R\}$ via:
\begin{align*}
  (g_1,g_2) \cdot (i, L) &= (g_1(i), L), \\
  (g_1,g_2) \cdot (j, R) &= (g_2(j), R).
\end{align*}
By abuse of notation, we identify the above action of $G$ on $[d]\times \{L,R\}$ with the resulting permutation group.
For $M \in \invp$ and $\Gamma$ associated to $M$ as in \Cref{subsec:graph}, we then view $\Aut_\iota(\Gamma)$ as a subgroup of $G$. In fact we have the following.

\begin{lemma}
For $g \in G$, $g(f_M) = f_M \Leftrightarrow g \in \Aut_\iota(\Gamma)$, where
$\Gamma$ is defined as in \Cref{subsec:graph}.
\end{lemma}
\begin{proof}
    We first decompose $f_M$ into three components: $f_M$ is a quadratic polynomial in the variables $x_1,\dots x_d, y_1, \dots y_d$. We can decompose
    $f_M = f_{M,2,0} + f_{M,1,1} + f_{M,0,2}$ by the multidegree in $x_1, \dots x_d$ and $y_1, \dots y_d$. The action of $G$ preserves multidegrees, so $g \in G$ fixes $f_M$ if it fixes all summands in the multidegree decomposition.

    Let us first show that $g(f_{M,2,0}) = f_{M,2,0}$ for $g = (g_1, g_2)$ if and only if $g_1$ commutes with $\phi$. Indeed, we have
    $f_{M,2,0} = \frac{1}{d}\sum_{1\leq i,j \leq d} M_{ij} x_i x_{\phi(i)} = \frac{1}{d}\sum_{1 \leq i \leq d} x_i x_{\phi(i)}$ since $M$ is doubly stochastic. So the coefficient of $x_i^2$ is $\frac{1}{d}$ if $i$ is a fixed point of $\phi$, and $0$ otherwise. 
    For $x_i x_j$ with $i \neq j$ the coefficient is $
    \frac{2}{d}$ if $i$ and $j$ are interchanged by $\phi$, and $0$ otherwise. So $g(f_{M,2,0}) = f_{M,2,0}$ if and only if $g_1$ sends fixed points of $\phi$ to fixed points of $\phi$, and pairs interchanged by $\phi$ to pairs interchanged by $\phi$; namely, if and only if $g_1$ commutes with $\phi$.
    We can see similarly that $g(f_{M,0,2}) = f_{M,0,2}$ if and only if $g_2$ commutes with $\psi$.

    Suppose now that $g \in G$ commutes with $\iota$. Let us show that $g(f_{M,1,1}) = f_{M,1,1}$ if and only if $g \in \Aut_\iota(\Gamma)$.
    Since $M\in \invp$, we can write
    \begin{equation*}
        f_{M,1,1} = - \frac{1}{d} \sum_{1\leq i,j \leq d} M_{ij} (x_i y_{\psi(j)} + x_{\phi(i)} y_j) =  - \frac{1}{d} \left( \sum_{1\leq i,j \leq d} M_{ij} x_i y_{\psi(j)} +  \sum_{1\leq i,j \leq d} M_{ij} x_{\phi(i)} y_j \right).
    \end{equation*}
    Now we can substitute $j \mapsto \psi(j)$ in the first sum and $i \mapsto \phi(i)$ in the second sum
    to obtain
    \begin{equation*}
        f_{M,1,1} = - \frac{1}{d} \sum_{1\leq i,j \leq d} (M_{i\psi(j)} + M_{\phi(i)j}) x_i y_j = - \frac{2}{d} \sum_{1\leq i,j \leq d} M_{i\psi(j)} x_i y_j ,
    \end{equation*}
    where the last equality comes from $M \in \invp$.
    Hence, for $g = (g_1, g_2) \in G$, we have
    \begin{equation*}
        g(f_{M,1,1}) = - \frac{2}{d} \sum_{1\leq i,j \leq d} M_{i\psi(j)} x_{g_1(i)} y_{g_2(j)} = - \frac{2}{d} \sum_{1\leq i,j \leq d} M_{g_1^{-1}(i)\psi(g_2^{-1}(j))} x_i y_j.
    \end{equation*}
    Therefore, $g(f_{M,1,1}) = f_{M,1,1}$ if and only if $ M_{i\psi(j)} = M_{g_1^{-1}(i)\psi(g_2^{-1}(j))}$ for all $i, j$. Since $g$ commutes with $\iota$, then $\psi$ commutes with $g^{-1}_2$. We can substitute $j \mapsto \psi(j)$ to see that $g(f_{M,1,1})=f_{M,1,1}$ is equivalent to $g$ preserving $\Gamma$.
\end{proof}
Let us denote
\begin{equation}\label{eqn:galois_orbit}
    h_M \coloneqq \prod_{g \in G/\Aut_\iota(\Gamma)} (g(f_M) - t) \in A[t].
\end{equation}
By \Cref{lem:galois_conjugates}, $h_M$ is irreducible in $\Q(x_1,\dots,x_d,y_1,\dots y_d)[t]$ and it has degree
$\frac{(d!)^2}{|\Aut_\iota(\Gamma)|}$.
We obtain the following result.
\begin{theorem}\label{thm:algebraic_degree}
    Let $p,q \in \Q[z]$ be monic polynomials of degree $d$, each with distinct roots. Let $\phi, \psi, \iota$ be as above. Suppose $M \in B^\iota_n$ is a minimizing vertex for $\cost(M, p, q)$
    and $\Gamma$ is the graph associated to $M$.
    Let $\Phi_{p,q} : \Q[x_1, \dots, x_d, y_1, \dots, y_d, t]  \rightarrow \C[t]$ be the extension of the map defined in \Cref{lem:specialize}, which by an abuse of notation we also denote it by $\Phi_{p,q}$.
    Then, $\Phi_{p,q}(h_M) \in \Q[t]$ and $\Wd(p,q)$ is a root of $\Phi_{p,q}(h_M)$.

    In particular, the algebraic degree of $\Wd(p,q)$ is bounded by
    $\frac{(d!)^2}{|\Aut_\iota(\Gamma)|}$.
    \label{thm:upper_bound}
\end{theorem}
\begin{proof}
    By Vieta's formulae, $\Phi_{p,q}$ maps $\sigma_i$ and $\tau_j$ (up to sign) to coefficients of $p$ and $q$. So $\Phi_{p,q}(A^G) \subset \Q$. The rest follows from the above construction.
\end{proof}

Let us show that this bound is sharp on a large set. Recall the open set of polynomial pairs $U_{(\phi,\psi)}^M$ defined in \Cref{prop:realization}, where polynomial roots are ordered satisfying the involution $\phi,\psi$ and $M$ is the unique minimizer for the cost function. 

\begin{theorem}\label{thm: bound sharp on dense set}
Let $\phi,\psi$ be a pair of involutions of $\Sym(d)$, $\iota=(\phi,\psi)$ be the involution of $\Sym(d)\times \Sym(d)$ and $M \in B^\iota_d$ be a vertex. 
There exists a dense set $U\subseteq U_{(\phi,\psi)}^M$ of polynomial pairs $(p, q)$ such that $\Phi_{p,q}(h_M)$ is irreducible.
In particular, the algebraic degree of $\Wd(p,q)$ is equal to $\frac{(d!)^2}{|\Aut_\iota(\Gamma)|}$ for these polynomials, so the bound of \Cref{thm:upper_bound} is sharp.
\end{theorem}
\begin{proof}
    The extended map $\Phi_{p,q} : \Q[\sigma_1, \dots, \sigma_d, \tau_1, \dots, \tau_d,t] \rightarrow \Q[t]$, is the specialization map sending $\sigma_i$ and $\tau_j$ (up to sign) to the coefficients of $p$ and $q$. We know that $h_M \in A^G[t]$ is irreducible by \Cref{lem:galois_conjugates}.  
    By Hilbert's irreducibility theorem (see e.g. \cite[Theorem 46]{Schinzel00:PolyRed} or \cite[Chapter 9, Corollary 2.5]{Lang83:FunDiophantine}),
    the specialization $\Phi_{p,q}(h_M)$ is irreducible for a dense set of monic polynomials $p,q \in \Q[z]$. In particular, on the open set $U_{(\phi,\psi)}^M$ that admits $M$ as the unique minimizer, we have that $\Phi_{p,q}(h_M)$ is the minimal polynomial of $\Wd(p,q)$ for a dense open subset $U$ of $U_{(\phi,\psi)}^M$. Hence, our bound in \Cref{thm:upper_bound} is sharp on $U$.
\end{proof}

\begin{example}
Suppose $p,q\in \mathbb{Q}[z]$ are generic monic polynomials of degree $d$, both with only real roots. The only possible type (cf. Theorem~\ref{thm: vertices of gamma}) for the disjoint cycles of the graph associated to vertices of $B_d^{\iota}$ is type $1$. Therefore, we get $c_1=d$, $c_2=c_{3,k}=c_{4,k,L}=c_{4,k,R}=0$ and $|\Aut_\iota(\Gamma)|=c_1!=d!$.
The algebraic degree of $W_2^2(p,q)$ is 
\[
\Wdeg(p,q) = \frac{(d!)^2}{d!}=d!,
\]
for generic $p,q$, by \Cref{thm: bound sharp on dense set}.
\end{example}

For degree $1$ polynomials, $\Wdeg = 1$ trivially. When $d = 2$, if $(p_\R,p_\C) = (q_\R,q_\C)$ and $p,q$ are generic, then $\Wdeg = 2$. In the remaining cases, the Wasserstein distance degree is $1$. The values of the Wasserstein distance degree for $d = 3,4$ are displayed in \Cref{table:WDdegrees}. Degree $4$ is the first interesting case where polynomials with the same number of real and complex roots give rise to multiple open sets where we get different algebraic degrees for $\Wd(p,q)$.
\begin{table}[ht]
    \centering
    {\def\arraystretch{1.3}
    \begin{tabular}{>{\centering\arraybackslash}m{1.7cm} |>{\centering\arraybackslash}m{3cm} |>{\centering\arraybackslash}m{3cm}  }
    degree & $(p_\R,p_\C), (q_\R,q_\C)$ & $\Wdeg(p,q)$ \\
    \hline
    \hline
    3 & (3,0), (3,0) & 6 \\
    3 & (3,0), (1,1) & 9 \\
    3 & (1,1), (1,1) & 18 \\
    4 & (4,0), (4,0) & 24 \\
    4 & (4,0), (2,1) & 72 \\
    4 & (4,0), (0,2) & 18 \\
    4 & (2,1), (2,1) & 36, 144, 288 \\
    4 & (2,1), (0,2) & 18, 288 \\
    4 & (0,2), (0,2) & 72 \\
    \end{tabular}
    }
    \caption{Values of $\Wdeg(p,q)$ for a dense set of pairs $(p,q)$.}
    \label{table:WDdegrees} 
\end{table}

Note that \Cref{thm:algebraic_degree} gives the algebraic degree of the squared Wasserstein distance for a dense set of pairs of polynomials. However, in special cases, it provides only an upper bound, as exhibited in the next example.

\begin{example}\label{ex:nongeneric}
    Consider $p = z^3-1$ with only one real root and $q = z^3 - 5z^2 + 4z + 3$ with three real roots. Let $\iota$ be the associated involution and $M$ the optimal vertex of $B^{\iota}_3$. \Cref{thm:algebraic_degree} predicts an upper bound $\Wdeg(p,q) \leq \deg h_M = 9$. Indeed, the minimal polynomial of the squared Wasserstein distance in this case is
    \[
    27t^3 - 540t^2 + 3483t - 7231,
    \]
    hence the Wasserstein distance degree is $3$.
\end{example}

\begin{remark}
    If $p,q \in \Q[z]$ (with possibly multiple roots), and $(p,q) \in \overline{U_{(\phi,\psi)}^M}$, then we still get that $\Wd(p,q)$ is a root of $\phi_{p,q}(h_M)$, so we still obtain a bound for the Wasserstein distance degree. However, this bound is far from being sharp in general. 
\end{remark}

\subsection{Wasserstein distance degree via elimination}\label{subsec:elimination}

While the Wasserstein distance degree is studied from combinatoric perspective, the algebraic degree of two specific polynomials can also be understood via elimination of ideals. We describe here how to compute the minimal polynomial of the Wasserstein distance of two given polynomials via elimination. 

Let $p,q \in \Q[z]$ be polynomials of degree $d$; we denote by $p_i,q_i$ the coefficient of the monomial $z^i$ in the univariate polynomial $p,q$ respectively. Consider the ideal in $\Q[x_1,\ldots,x_d,y_1,\ldots,y_d]$ generated by Vieta's formulae, i.e., all the relations between the elementary symmetric polynomials in the roots of $p,q$ and their coefficients:
\[
I = \left\langle \sum_{i=1}^d x_i + \frac{p_{d-1}}{p_d}, \ldots , \prod_{i=1}^d x_i - (-1)^d \frac{p_0}{p_d}, \sum_{i=1}^d y_i + \frac{q_{d-1}}{q_d}, \ldots , \prod_{i=1}^d y_i - (-1)^d \frac{q_0}{q_d} \right\rangle .
\]
Assume that the roots of $p,q$ are simple.
Then, the ideal $I$ is zero dimensional. The variety associated to it consists of all tuples of points in $\C^{2d}$ such that the first $d$ coordinates are the roots of $p$, and the last $d$ coordinates are the roots of $q$.
Counting all permutations of the two sets of coordinates, we get that $\deg I = (d!)^2$. Notice that in the case where $p,q$ are allowed to have double roots, the degree of $I$ decreases but $(d!)^2$ remains a valid upper bound.

Assume that the minimizer of the cost function $\cost(\cdot,p,q)$ is attained at the vertex $M$ of $\invp$, and let $h_M$ be the invariant irreducible polynomial defined in \eqref{eqn:galois_orbit}.
Consider a second ideal
\[
J = I +  \left\langle h_M(x,y,t) \right\rangle \subset \Q[x_1,\ldots,x_d, y_1,\ldots,y_d,t],
\]
that has the same dimension and degree as $I$.
Then, the elimination ideal $E = J \cap \Q[t]$ has one generator, namely $\Phi_{p,q}(h_M)(t)$, where $\Phi_{p,q}$ is the specialization map sending $x$ to the roots of $p$ and $y$ to the roots of $q$. During the elimination process or, more geometrically, during this projection, some of the $(d!)^2$ points in the variety of $J$ get mapped to the same point. The polynomial $\Phi_{p,q}(h_M)(t)$ obtained contains as a factor the minimal polynomial of $W_2^2(p,q)$. For a dense set of pairs $(p,q)$ it is actually irreducible and its degree is the Wasserstein distance degree of $p,q$.
 
The method described above is used to compute \Cref{ex:intro,ex:nongeneric}.

\paragraph{Acknowledgements.}
We would like to thank Bernd Sturmfels who made this work possible. We also thank Sarah-Tanja Hess for interesting discussions.

\small
\bibliographystyle{alpha}
\bibliography{biblio}

\vfill
\footnotesize{

\noindent \textsc{Chiara Meroni} \\
\textsc{ Harvard University \\
29 Oxford Street, Cambridge, MA 02138} \\
\url{cmeroni@seas.harvard.edu} \\

\noindent \textsc{Bernhard Reinke}\\
\textsc{Max Planck Institute for Mathematics in the Sciences,\\ 
Inselstrasse 22, Leipzig, 04107} \\
\url{bernhard.reinke@mis.mpg.de} \\

\noindent \textsc{Kexin Wang} \\
\textsc{ Harvard University \\
29 Oxford Street, Cambridge, MA 02138} \\
\url{kexin_wang@g.harvard.edu} \\
}

\end{document}